\documentclass[smallextended]{svjour3}     
\smartqed  % flush right qed marks, e.g. at end of proof
\usepackage[normalem]{ulem}
\usepackage[mathscr]{eucal}
\usepackage{blindtext}
\usepackage{hyperref}
\hypersetup{
	colorlinks=truerue,
	citecolor=blue,
	linkcolor=blue,
	filecolor=magenta,      
	urlcolor=blue,
}
\makeatletter
\newcommand{\neutralize}[1]{\expandafter\let\csname c@#1\endcsname\count@}
\makeatother

\usepackage{enumitem}
\usepackage{array}
\usepackage[toc,page]{appendix}
\usepackage{multicol, latexsym}
\usepackage{blindtext}
\usepackage{subcaption}
\usepackage{url}
\usepackage{enumitem}
\usepackage{bm}
\usepackage{float}
\usepackage{array}
\usepackage{cite}
\usepackage{chngcntr}
\usepackage{graphicx}
\usepackage{enumerate}
\usepackage{epstopdf}
\usepackage{multirow}
\usepackage{textcomp}
\usepackage{amsmath}
\usepackage{amssymb}
\usepackage{epsfig}
\usepackage{tabularx, caption}
\usepackage[misc]{ifsym}

\allowdisplaybreaks

\usepackage{latexsym}
\journalname{}
\newtheorem{assumption}{Assumption}
\begin{document}
\bibliographystyle{plainnat}
	\title{Limit theorems for filtered long-range dependent random fields}
	\author{T. Alodat\and N. Leonenko\and A. Olenko}
	
	\institute{T. Alodat
		\at	
		Department of Mathematics and Statistics, La Trobe University,
		\at	
		Melbourne, VIC, 3086, Australia
		\at
		\email{alodat.t@students.latrobe.edu.au}\\
		\and
		N. Leonenko
		\at	
		School of Mathematics, Cardiff University, Senghennydd Road, 
		\at
		Cardiff CF2 4YH, UK
		\at
		\email{leonenkon@Cardiff.ac.uk}\\
		\and \Letter\  A. Olenko
		\at	
		Department of Mathematics and Statistics, La Trobe University,
		\at	
		Melbourne, VIC, 3086, Australia \at
		\email{a.olenko@latrobe.edu.au}           %  \\
		%             \emph{Present address:} of F. Author  %  if needed\\
	}

	\date{}
	
	\maketitle

	\begin{abstract}
		This article investigates general scaling settings and limit distributions of functionals of filtered random fields. The filters are defined by the convolution of non-random kernels with functions of Gaussian random fields. The case of long-range dependent fields and increasing observation windows is studied. The obtained limit random processes are non-Gaussian. Most known results on this topic give asymptotic processes that always exhibit non-negative auto-correlation structures and have the self-similar parameter $H\in(\frac{1}{2},1)$. In this work we also obtain convergence for the case $H\in(0,\frac{1}{2})$ and show how the Hurst parameter $H$ can depend on the shape of the observation windows. Various examples are presented.
		\keywords{filtered random fields \and long-range dependence \and self-similar processes\and non-central limit theorem\and Hurst parameter}
		\subclass{60G60 \and 60F17 \and 60F05 \and 60G35}
	\end{abstract}
	\section{Introduction}
Over the last four decades, several studies dealt with various functionals of random fields and their asymptotic behaviour~\cite{anh2017rate,anh2015rate,bai2013multivariate,dobrushin1979non,doukhan2002theory,taqqu1979convergence,taqqu1975weak,ivanov2008semiparametric,leonenko2014sojourn,kratz2017central}. These functionals play an important role in various fields, such as physics, cosmology, telecommunications, just to name a few. In particular, asymptotic results were obtained either for integrals or additives functionals of random fields under long-range dependence, see~\cite{leonenko2006weak,weak2017alodat,olenko2010limit,olenko2013limit,doukhan2002theory,nourdin2014central,leonenko2017rosenblatt} and the references~therein.  

It is well known that functionals of Gaussian random fields with long-range dependence can have non-Gaussian asymptotics and require normalising factors different from those in central limit theorems. These limit processes are known as Hermite or Hermite-Rosenblatt processes. The first result in this direction was obtained in~\cite{rosenblatt1961independence} where quadratic functionals of long-range dependent stationary Gaussian sequences were investigated. The pioneering results in the asymptotic theory of non-linear functionals of long-range dependent Gaussian processes and sequences can be found in~\cite{taqqu1979convergence,taqqu1975weak,dobrushin1979non,rosenblatt1981limit,taqqu1978representation}. This line of studies attracted much attention, for example, in~\cite{pakkanen2016functional} it was shown that the limiting distribution of generalised variations of a long-range dependent fractional Brownian sheet is a fractional Brownian sheet that is independent and different from the original one. Some statistical properties of the Rosenblatt distribution, as well as its expansion in terms of shifted chi-squared distributions were studied in~\cite{veillette2013}. The L\'{e}vy-Khintchine formula and asymptotic properties of the L\'{e}vy measure, were also addressed in~\cite{leonenko2017rosenblatt}. Some weighted functionals for long-range dependent random fields were considered and limit theorems were investigated in a number of papers, including~\cite{olenko2013limit,ivanov2013limit,ivanov1989statistical}.

Linear stochastic processes and random fields obtained as outputs of filters are popular models in various applications, see~\cite{jazwinski1970stochastic,wiener1949extrapolation,kallianpur2013stochastic,alomari2018estimation}. In engineering practice it is often assumed that a narrow band-pass filter applied to a stationary random input yields an approximately normally distributed output. Of course, such results are not true in general, especially when the stationary input has some singularity in the spectrum and the linear filtration is replaced by a non-linear one.

We recall the classical central-limit type theorem by Davydov~\cite{davydov1970invariance} for discrete time linear stochastic processes. 

\begin{theorem}{\rm\cite{davydov1970invariance}}
	Let $V(t)=\sum_{j\in\mathbb{Z}}G_{t-j}\xi_{j},\ t\in\mathbb{Z}$, where ${\xi_{j}}$ is a sequence of i.i.d random variables with zero mean and finite variance {\rm(the $\{\xi_{j}\}$ are not necessarily Gaussian)}. Suppose that ${G_{j}}$ is a real-valued sequence satisfying $\sum_{j\in\mathbb{Z}}G_{j}^{2}<\infty$ and let $X_{r}^{(d)}:=\sum_{t=1}^{r}V(t)$. If $Var X_{r}^{(d)}=r^{2H}\mathcal{L}^{2}(r)$ as $r\to\infty$, where $H\in(0,1)$ and the function $\mathcal{L}(\cdot)$ is a slowly varying at infinity, then
	\[
	X_{r}^{(d)}(t)=\frac{1}{r^{H}\mathcal{L}(r)}\sum_{s=1}^{[rt]}V(s)\stackrel{D}{\rightarrow}B_{H}(t),\ t>0,\ \text{as}\ r\to\infty,
	\]  
	in the sense of convergence of finite-dimensional distributions, where $B_{H}(t)$, $t>0$, is the fractional Brownian motion with zero mean and the covariance function
	$B_{H}(t,s)=\frac{1}{2}\left(\vert t\vert^{2H}+\vert s\vert^{2H}-\vert t-s\vert^{2H}\right),\ t,s>0,\ 0<H<1.$
	\end{theorem}

One can obtain an analogous result for the case of continuous time.
\begin{theorem}
	Let $V(t)=\int_{\mathbb{R}}G(t-s)\xi(s)ds,\ t\in\mathbb{R}$, be a linear filtered process, where $\xi(t),\ t\in\mathbb{R}$, be a mean-square continuous stationary in the wide sense process with zero mean and finite variance. Suppose that $G(t),\ t\in\mathbb{R}$, is a non-random function, such that $\int_{\mathbb{R}}G^{2}(t)dt<\infty$. Let $X_{r}^{(c)}:=\int_{0}^{r}V(s)ds$. If $Var X_{r}^{(c)}=r^{2H}\mathcal{L}^{2}(r),$ as $r\to\infty$, where $H\in(0,1)$ and $\mathcal{L}(\cdot)$ is slowly varying at infinity, then
	\[
	X_{r}^{(c)}(t)=\frac{1}{r^{H}\mathcal{L}(r)}\int_{0}^{rt}V(s)ds\stackrel{D}{\rightarrow}B_{H}(t),\ t>0,\ \text{as}\ r\to\infty,
	\]
	in a sense of convergence of finite-dimensional distributions. 	
\end{theorem}

The equivalence of the statements for the discrete and continuous time follows from the results in Leonenko and Taufer~\cite{leonenko2006weak} and Alodat and Olenko \cite{weak2017alodat}.   

It was Rossenblatt~\cite{Rosenblatt1979} (see also Major~\cite{major1981limit}, Taqqu~\cite{doukhan2002theory}) who first proved that for a discrete-time Gaussian stochastic process $\{\xi_{j},\ j\in\mathbb{Z}\}$, with zero mean and long-range dependence and the $\kappa$-th Hermite polynomials $H_{\kappa}$, the non-linear filtered process 
\[
V_{\kappa}(t)=\sum_{j\in\mathbb{Z}}G_{t-j}H_{\kappa}(\xi_{j}),
\] 
satisfies the non-central limit theorem, that is for some normalising $A_{r}$ it holds
\[
\frac{1}{A_{r}}\sum_{s=1}^{[rt]}V_{\kappa}(s)\stackrel{D}{\rightarrow}Y_{\kappa}(t),\ t>0,\ \text{as}\ r\to\infty,
\]
where $Y_{\kappa}(t)$ is a self-similar process with the Hurst parameter $H\in(0,1)$ (non-Gaussian, if $\kappa\geq2$).

The limit processes $Y_{\kappa}(t),\ t>0$, are given in terms of $\kappa$-fold Wiener-It{\^o} stochastic integrals, and are the fractional Brownian motions with the Hurst parameter $H\in(0,1)$ if $\kappa=1$. 

The aim of this paper is to give an extension of the results of Rossenblatt\cite{Rosenblatt1979}, Major~\cite{major1981limit}, Taqqu~\cite{doukhan2002theory} for the case of random fields. Motivated by the theory of renormalisation and homogenisation of solutions of randomly initialised partial differential equations (PDE) and fractional partial differential equations (FPDE) (see, e.g.~\cite{albeverio1994stratified,leonenko1998exact,liu2010scaling,leonenko1998scaling}), we study the asymptotic behaviour of integrals of the form 
\[
d_{r}^{-1}\int_{\Delta(rt^{1/n})}V(x)dx,\quad t\in[0,1],\ \text{as}\ r\rightarrow\infty,
\]
where $V(x),\ x\in\mathbb{R}^{n}$, is a random field, $\Delta\subset\mathbb{R}^{n}$ is an observation window and $d_{r}$ is a normalising factor. The case when the limit process is self-similar with parameter $H\in(0,1)$ is considered.

The parameter $H$ plays an important role in analysing stochastic processes and can be used for their classification. In particular, stochastic processes can be classified according to the range of $H$ to the Brownian motion $(H=0.5)$, a short-memory anti-persistent stochastic process $(H\in(0,0.5))$ and a long-memory stochastic process ($H\in(0.5,1)$). These three cases correspond to the three types of behaviour called  $\frac{1}{f}$ noise, ultraviolet and infrared catastrophes by Mandelbrot and Taqqu~\cite{taqqu1981self}. The literature shows a variety of limit theorems with asymptotics given by non-Gaussian self-similar processes that exhibit non-negative auto-correlation structures with parameter $H\in(0.5,1)$, see~\cite{taqqu1975weak,taqqu1979convergence,ivanov1989statistical,leonenko1999limit,olenko2013limit,olenko2010limit} and references therein. However, there are only few results where asymptotic processes have $H\in(0,0.5)$. In the case $H<0.5$ processes exhibit a negative dependence structure, which is useful in applied modelling of switching between high and low values. Also, such processes have interesting theoretical stochastic properties. For example, in this case the covariance is the Green function of a Markov process and the squared process is infinitely divisible, which is not true for the case $H>0.5$, see\cite{eisenbaum2005squared,eisenbaum2006characterization}.

The example of a non-Gaussian self-similar process with $H\in(0,0.5)$ was given by Rossenblatt~\cite{Rosenblatt1979} where the asymptotic of quadratic functions of a long-range Gaussian stationary sequence was investigated. The result was generalised in~\cite{major1981limit} for sums of non-linear functionals of Gaussian sequences. In this paper we extend these results in several directions for more general conditions and derive limit theorems for functionals of filtered random fields defined as the convolution
\begin{equation*}
V(x):=\int_{\mathbb{R}^{n}}G(\Vert y-x\Vert)S(\xi(y))dy,
\end{equation*}
where $G(\cdot)$, $S(\cdot)$ are non-random functions and $\xi(\cdot)$ is a long-range dependent random field.

In the limit theorems obtained in this paper the asymptotic processes have the self-similar parameters $H\in\left(\gamma(\Delta),1\right)$, where $\gamma(\Delta)\geq0$ depends on the geometry of the set $\Delta\subset\mathbb{R}^{n}$. In the one-dimensional case $d=1$, $\gamma(\Delta)=0$ which coincides with the known results in the literature.

The rest of the article is organised as follows. In Section 2 we outline the necessary background. In Section 3 we introduce assumptions and give auxiliary results from the spectral and correlation theory of random fields. In Section 4 we present main results on the asymptotic behaviour of functionals of filtered random fields. Examples are presented in Section~5.	
\section{Notations}
This section gives main definitions and notations that are used in this paper.

In what follows $|\cdot|$ and $\Vert \cdot \Vert$ are used for the Lebesque measure and the Euclidean distance in $\mathbb{R}^{n},\ n\geq 1$, respectively. The symbols $C$, $\epsilon$ and $\delta$ (with subscripts) will be used to denote constants that are not important for our discussion. Moreover, the same symbol may be used for different constants appearing in the same proof.
 \begin{definition}
	A real-valued function $h:[0,\infty)\rightarrow\mathbb{R}$ is homogeneous of degree $\beta$ if $h(ax)=a^{\beta}h(x)$ for all $a,\ x>0$.
\end{definition}
 \begin{definition}\label{def3}{\rm\cite{bingham1987regular}} A measurable function $\mathcal{L}:(0,\infty)\rightarrow (0,\infty)$ is slowly varying at infinity if for all $t>0$,
	$$\lim_{r\rightarrow \infty} \dfrac{\mathcal{L}(tr)}{\mathcal{L}(r)} =1.$$
\end{definition}

By the representation theorem~\cite[Theorem 1.3.1]{bingham1987regular}, there exists $C>0$ such that for all $r\geq C$ the function $\mathcal{L}(\cdot)$ can be written in the form
\begin{align*}
\mathcal{L}(r)=\exp\left(\zeta_{1}(r)+\int_{C}^{r}\dfrac{\zeta_{2}(u)}{u}du\right),
\end{align*}
where $\zeta_{1}(\cdot)$ and $\zeta_{2}(\cdot)$ are  such measurable and bounded functions that $\zeta_{2}(r)\rightarrow 0$ and $\zeta_{1}(r)\rightarrow C_{0}$, $\left(C_{0}<\infty\right)$, when $r\rightarrow\infty$.

If $\mathcal{L}(\cdot)$ varies slowly, then $r^{a}\mathcal{L}(r)\rightarrow\infty$, and $r^{-a}\mathcal{L}(r)\rightarrow 0$ for an arbitrary $a>0$ when $r\rightarrow\infty$, see Proposition 1.3.6~\cite{bingham1987regular}.
 \begin{definition}{\rm\cite{bingham1987regular}}
	A measurable function $g:(0,\infty)\rightarrow (0,\infty)$ is regularly varying at infinity, denoted $g(\cdot)\in R_{\tau}$, if there exists $\tau$ such that, for all $t>0$, it holds that
	\begin{align*}
	\lim_{r\rightarrow \infty} \dfrac{g(tr)}{g(r)} =t^{\tau}.
	\end{align*}
\end{definition}

\begin{theorem}\label{the1new}{\rm\cite[Theorem 1.5.3]{bingham1987regular}}
	Let $g(\cdot)\in R_{\tau}$, and choose $a\geq0$ so that $g$ is locally bounded on $[a,\infty)$. If $\tau>0$ then
	$$\sup_{a\leq t\leq x} g(t)\sim g(x)\ {\rm{and}}\ \inf_{t\geq x} g(t)\sim g(x),\quad x\rightarrow\infty.$$
	If $\tau<0$ then
	$$\sup_{t\geq x} g(t)\sim g(x)\ {\rm{and}}\ \inf_{a\leq t\leq x} g(t)\sim g(x),\quad x\rightarrow\infty.$$
\end{theorem}
\begin{definition}
	The Hermite polynomials $H_{m}(x),\ m\geq 0$, are given by
	
	\begin{equation*}
	H_{m}(x):=(-1)^{m}\exp\left(\dfrac{x^{2}}{2}\right)\dfrac{d^{m}}{dx^{m}}\exp\left(-\dfrac{x^{2}}{2}\right).
	\end{equation*}
\end{definition}

The first few Hermite polynomials are
\[
H_{0}(x)=1,\  H_{1}(x)=x,\  H_{2}(x)=x^{2}-1,\ H_{3}(x)=x^{3}-3x.
\]

The Hermite polynomials $H_{m}(x),\ m\geq 0$, form a complete orthogonal system in the Hilbert space $L_{2}\left(\mathbb{R},\phi(\omega)d\omega\right)=\left\{S: \int_{\mathbb{R}}S^{2}(\omega)\phi(\omega)d\omega<\infty\right\}$, where 
$\phi(\omega)$ is the probability density function of the standard normal distribution.

An arbitrary function $S(\omega)\in L_{2}\left(\mathbb{R},\phi(\omega)d\omega\right)$ possesses the mean-square convergent expansion
\begin{equation}\label{eq2}
S(\omega)=\sum_{j=0}^{\infty}\dfrac{C_{j}H_{j}(\omega)}{j!},\quad C_{j}:=\int_{\mathbb{R}}S(\omega)H_{j}(\omega)\phi(\omega)d\omega.
 \end{equation}
 By Parseval's identity 
 \begin{equation*}
 \sum_{j=0}^{\infty}\dfrac{C_{j}^{2}}{j!}=\int_{\mathbb{R}}S^{2}(\omega)\phi(\omega)d\omega.
 \end{equation*}
 \begin{definition}\label{def1}{\rm\cite{taqqu1975weak}} Let $S(\omega)\in L_{2}\left(\mathbb{R},\phi(\omega)d\omega\right)$ and there exists an integer $\kappa\geqslant 1$, such that $C_{j}=0$ for all $0<j\leq \kappa-1$, but $C_{\kappa}\neq 0$. Then $\kappa$ is called the Hermite rank of $S(\cdot)$ and is denoted by $H rankS(\cdot).$
 \end{definition}

It is assumed that all random variables are defined on a fixed probability space $\left(\Omega,\mathfrak{F},\mathbb{P}\right)$. We consider a measurable mean-square continuous zero-mean homogeneous isotropic real-valued random field $\xi\left(x\right),\ x\in\mathbb{R}^{n}$, with the covariance function
\begin{equation*}
B\left(r\right):=\mathbb{E}\left(\xi(0)\xi(x)\right),\quad x\in\mathbb{R}^{n},\quad
r=\Vert x\Vert.
\end{equation*}

It is well known that there exists a bounded non-decreasing function $\Phi\left(u\right)$, $u\geq 0$, (see~\cite{ivanov1989statistical,yadrenko1983spectral}) such that
\begin{equation*}
B\left(r\right)=\int_{0}^{\infty}Y_{n}\left(ru\right)d\Phi\left(u\right),
\end{equation*}
where the function $Y_{n}\left(\cdot\right),\ n\geq 1,$ is defined by
\begin{equation*}
Y_{n}\left(u\right):=2^{(n-2)/2}\Gamma\left(\dfrac{n}{2}\right)J_{(n-2)/2}(u)u^{(2-n)/2},\quad u\geqslant 0,
\end{equation*}
where $J_{(n-2)/2}(\cdot)$ is the Bessel function of the first kind of order $(n-2)/2$, see~\cite{leonenko1999limit,yadrenko1983spectral}. The function $\Phi\left(\cdot\right)$ is called the isotropic spectral measure of the random field $\xi\left(x\right),\ x\in\mathbb{R}^{n}$.

\begin{definition}
	If there exists a function $f(u),\ u\in[0,\infty)$, such that
	\begin{equation*}
	u^{n-1}f(u)\in L_{1}([0,\infty)),\quad\Phi(u)=2\pi^{n/2}/\Gamma(n/2)\int_{0}^{u}z^{n-1}f(z)dz,
	\end{equation*}
	then the function $f(\cdot)$ is called the isotropic spectral density of the field $\xi\left(x\right)$.
\end{definition}

The field $\xi\left(x\right)$ with an absolutely continuous spectrum has the following isonormal spectral representation
\begin{equation}\label{eq2xi}
\xi\left(x\right)=\int_{\mathbb{R}^{n}}e^{i\langle\lambda,x \rangle}\sqrt{f(\Vert\lambda\Vert)}W(d\lambda),
\end{equation}
where $W(\cdot)$ is the complex Gaussian white noise random measure on $\mathbb{R}^{n}$, see~\cite{ivanov1989statistical,leonenko1999limit,yadrenko1983spectral}.

Note, that by (2.1.8)~\cite{leonenko1999limit} we get $\mathbb{E}\left(H_{m}(\xi(x))\right)=0$ and
\[
\mathbb{E}\left(H_{m_{1}}(\xi(x))H_{m_{2}}(\xi(y))\right)=\delta_{m_{1}}^{m_{2}}m_{1}!B^{m_{1}}(\Vert x-y\Vert),\quad x, y\in\mathbb{R}^{n},
\]
where $\delta_{m_{1}}^{m_{2}}$ is the Kronecker delta function.

\begin{definition}
	A random process $X(t),\ t>0$, is called self-similar
	with parameter $H>0$, if for any $a>0$ it holds $X(at)\stackrel{D}{=}a^{H}X(t)$.
\end{definition}

 If $X(t),\ t>0$, is a self-similar process with parameter $H>0$ such that $\mathbb{E}(X(t))=0$ and $\mathbb{E}(X^{2}(t))<\infty$, then $B(at,as)=a^{2H}B(t,s)$, see \cite{leonenko1999limit}.
 
 \section{Assumptions and auxiliary results}
 This section introduces assumptions and results from the spectral and correlation theory of random fields.
 
 \begin{assumption}\label{ass1}
 	Let $\xi(x),\ x\in\mathbb{R}^{n}$, be a homogeneous isotropic Gaussian random field with $\mathbb{E}\xi(x)=0$ and the covariance function $B(x)$, such that $B(0)=1$ and
 	$$B(x)=\mathbb{E}\left(\xi\left(0\right)\xi\left(x\right)\right)=\Vert x\Vert^{-\alpha}\mathcal{L}_{0}\left(\Vert x\Vert\right),\quad \alpha>0,$$\
 	where $\mathcal{L}_{0}\left(\Vert\cdot\Vert\right)$ is a function slowly varying at infinity.  
 \end{assumption}
 If $\alpha\in\left(0,n\right)$, then the covariance function $B(x)$ satisfying Assumption \ref{ass1} is not integrable, which corresponds to the long-range dependence case\cite{anh2015rate}.
 
 The notation $\Delta\subset\mathbb{R}^{n}$ will be used to denote a Jordan-measurable compact bounded set, such that $|\Delta|>0$, and $\Delta$ contains the origin in its interior. Let $\Delta(r),\ r>0$, be the homothetic image of the set $\Delta$, with the centre of homothety at the origin and the coefficient $r>0$, that is $\vert\Delta(r)\vert=r^{n}\vert\Delta\vert$ and $\Delta=\Delta(1)$.
 
 Let $S(\omega)\in L_{2}\left(\mathbb{R},\phi(\omega)d\omega\right)$ and denote the random variables $K_{\kappa}$ and $K_{r,\kappa}$ by
 $$K_{\kappa}:=\int_{\bigtriangleup(r)}S\left(\xi\left(x\right)\right)dx\quad {\rm{and}}\quad K_{r,\kappa}:=\dfrac{C_{\kappa}}{\kappa!}\int_{\bigtriangleup(r)}H_{\kappa}\left(\xi\left(x\right)\right)dx,$$
 where $C_{\kappa}$ is given by (\ref{eq2}).
 \begin{theorem}\label{theo1}{\rm\cite{leonenko2014sojourn}}
 	Suppose that $\xi\left(x\right),\ x\in\mathbb{R}^{n}$, satisfies Assumption {\rm\ref{ass1}} and $H rankS(\cdot)=\kappa\geq 1$. If a limit distribution exists for at least one of the random variables
 	$$\dfrac{K_{r}}{\sqrt{Var K_{r}}}\quad and\quad \dfrac{K_{r,\kappa}}{\sqrt{Var K_{r,\kappa}}},$$
 	then the limit distribution of the other random variable also exists, and the limit distributions coincide when $r\rightarrow \infty$. 
 \end{theorem}
By Theorem \ref{theo1} it is enough to study $K_{r,\kappa}$ to get asymptotic distributions of $K_{\kappa}$. Therefore, we restrict our attention only to $K_{r,\kappa}$.
\begin{assumption}\label{ass2}
	The random field $\xi\left(x\right),\ x\in\mathbb{R}^{n}$, has the isotropic spectral density
	\begin{equation*}
	f\left(\Vert \lambda\Vert\right)=c_{1}\left(n,\alpha\right)\Vert\lambda\Vert^{\alpha-n}\mathcal{L}\left(\dfrac{1}{\Vert  \lambda\Vert}\right),
	\end{equation*}
	where $\alpha\in(0,n),\ c_{1}\left(n,\alpha\right):=\Gamma\left(\frac{n-\alpha}{2}\right)/2^{\alpha}\pi^{n/2}\Gamma\left(\frac{\alpha}{2}\right),$ and $\mathcal{L}(\Vert\cdot\Vert)\sim \mathcal{L}_{0}(\Vert\cdot\Vert)$ is a locally bounded function which is slowly varying at infinity.
\end{assumption}
One can find more details on relations between Assumptions~\ref{ass1} and \ref{ass2} in \cite{anh2017rate, anh2015rate}.

The function $K_{\Delta}\left(x\right)$ will be used to denote the Fourier transform of the indicator function of the set $\Delta$, i.e.
\begin{equation}\label{eq3}
K_{\Delta}\left(x\right):=\int_{\Delta}e^{i\langle u,x \rangle}du,\quad x\in\mathbb{R}^{n}.
\end{equation}  
\begin{theorem}\label{theo2}{\rm\cite{leonenko2014sojourn}}
	Let $\xi\left(x\right),\ x\in\mathbb{R}^{n}$, be a homogeneous isotropic Gaussian random field. If Assumptions {\rm\ref{ass1}} and {\rm\ref{ass2}} hold, then for $r \rightarrow\infty$ the random variables 
	$$X_{r,\kappa}(\Delta):=r^{\kappa\alpha/2-n}\mathcal{L}^{-\kappa/2}(r)\int_{\Delta(r)}H_{\kappa}\left(\xi(x)\right)dx$$\
	converge weakly to  
	\[
	X_{\kappa}(\Delta):=c_{1}^{\kappa/2}(n,\alpha)\int_{\mathbb{R}^{n\kappa}}^{\prime}K_{\Delta}\left(\lambda_{1}+\cdots+\lambda_{\kappa}\right)\dfrac{W(d\lambda_{1})\cdots W(d\lambda_{\kappa})}{\Vert\lambda_{1}\Vert^{(n-\alpha)/2}\cdots\Vert\lambda_{\kappa}\Vert^{(n-\alpha)/2}}.
	\]
\end{theorem}
Here $\int_{\mathbb{R}^{n\kappa}}^{\prime}$ denotes the multiple Wiener-It{\^o} integral with respect to a Gaussian white noise measure, where the diagonal hyperplanes $\lambda_{i}=\pm\lambda_{j},\ i,j=1,\dots,\kappa,\ i\neq j$, are excluded from the domain of integration.

\begin{assumption}\label{ass2weighted}{\rm\cite{ivanov1989statistical}} 
	Let $\vartheta(x)=\vartheta(\Vert x\Vert)$ be a radial continuous function positive for $\Vert x\Vert>0$ and such that there exists
	\begin{align*}
	\tilde{C}:=\kappa!\lim_{r\rightarrow \infty}\int_{\Delta}\int_{\Delta}\dfrac{\vartheta(r\Vert x\Vert)\vartheta(r\Vert y\Vert)dxdy}{\vartheta^{2}(r)\Vert x-y\Vert^{\alpha\kappa}}\in\left(0,\infty\right),
	\end{align*}
	where $\alpha\in\left(0,n/\kappa\right)$.
\end{assumption}

Let $u(\Vert\lambda\Vert):=c_{1}\left(n,\alpha\right)\mathcal{L}\left(\dfrac{1}{\Vert \lambda\Vert}\right)$, where $\mathcal{L}(\cdot)$ is from Assumption \ref{ass2}. In~\cite{leonenko1999limit} and Section 10.2~\cite{ivanov1989statistical} the case when the function $u(\Vert\lambda\Vert)$ is continuous in a neighborhood of zero, bounded on $ (0,\infty) $  and $u(0)\neq0$, was studied. It was assumed that there is a function $\bar{\vartheta}(\Vert x\Vert)$ such that
\begin{align*}
\int_{\mathbb{R}^{n\kappa}}\prod_{j=1}^{\kappa}\Vert \lambda_{j}\Vert^{\alpha-n}\bigg\vert\int_{\Delta(t^{1/n})}e^{i\langle \lambda_{1}+\dots+\lambda_{\kappa},x \rangle}\bar{\vartheta}(x)dx\bigg\vert^{2}\prod_{j=1}^{\kappa}d\lambda_{j}<\infty
\end{align*}
and
\begin{align*}
&\lim_{r\rightarrow \infty}\int_{\mathbb{R}^{n\kappa}}\bigg\vert\int_{\Delta(t^{1/n})}e^{i\langle \lambda_{1}+\dots+\lambda_{\kappa},x \rangle}
	\left(\dfrac{\vartheta(r\Vert x \Vert)}{\vartheta(r)}\prod_{j=1}^{\kappa}\sqrt{\dfrac{u(\Vert\lambda_{j}\Vert r^{-1})}{u(0)}}-\bar{\vartheta}(x)\right)dx\bigg\vert^{2}\\
&\times\prod_{j=1}^{\kappa}\Vert \lambda_{j}\Vert^{\alpha-n}\prod_{j=1}^{\kappa}d\lambda_{j}=0,
\end{align*}
for all $t\in[0,1]$.

Under these assumptions the following result was obtained.
\begin{theorem}\label{theo3}{\rm\cite{ivanov1989statistical}}
	If Assumption {\rm\ref{ass2weighted}} holds, then the finite-dimensional distributions of the random processes
	\begin{align}\label{eq6new}
	Y_{r,\kappa}(t):=\left(r^{n-\kappa\alpha/2}\vartheta(r)\sqrt{\tilde{C}c_{1}^{\kappa}(n,\alpha)u^{\kappa}(0)}\right)^{-1}\int_{\Delta(rt^{1/n})}\vartheta(\Vert x\Vert)H_{\kappa}\left(\xi(x)\right)dx
	\end{align}
	
	converge weakly to finite-dimensional distributions of the processes  
	$$Y_{\kappa}(t):=\dfrac{1}{\sqrt{\tilde{C}c_{1}^{\kappa}(n,\alpha)}}\int_{\mathbb{R}^{n\kappa}}^{\prime}K_{\Delta(t^{1/n})}\left(\lambda_{1}+\cdots+\lambda_{\kappa};\bar{\vartheta}\right)\dfrac{\prod_{j=1}^{\kappa} W(d\lambda_{j})}{\prod_{j=1}^{\kappa}\Vert\lambda_{j}\Vert^{(n-\alpha)/2}},$$	
	as $r\to\infty$, where
	$\alpha\in\left(0,\min\left(\frac{n}{\kappa},\frac{n+1}{2}\right)\right)$
	and
	$$
	K_{\Delta(t^{1/n})}\left(\lambda;\bar{\vartheta}\right):=\int_{\Delta(t^{1/n})}e^{i\langle \lambda,x \rangle}\bar{\vartheta}(x)dx.
	$$
\end{theorem}

\section{Limit theorems for functionals of filtered fields}
This section derives the generalisation of Theorem~\ref{theo3} when the integrand $\vartheta(\cdot)H_{\kappa}\left(\cdot\right)$ in~(\ref{eq6new}) is replaced by a filtered random field.

\begin{assumption}\label{ass4}
	Let $h:[0,\infty)\rightarrow\mathbb{R}$ be a measurable real-valued homogeneous function of degree~$\beta$ and $g:[0,\infty)\rightarrow\mathbb{R}$ be a bounded uniformly continuous function such that $g(0)\neq 0$ in some neighberhood of zero and $\int_{\mathbb{R}^{n}}h^{2}(\Vert u\Vert)g^{2}(\Vert u\Vert)du<\infty$.
\end{assumption}
We define the filtered random field $V(x),\ x\in\mathbb{R}^{n}$, as
\begin{equation}\label{eq5}
V(x)=\int_{\mathbb{R}^{n}}G(\Vert y-x\Vert)H_{\kappa}(\xi(y))dy=\int_{\mathbb{R}^{n}}G(\Vert y\Vert)H_{\kappa}(\xi(x+y))dy,
\end{equation}
where
\begin{equation}\label{eq6}
G(\Vert x\Vert):=\dfrac{1}{(2\pi)^{n}}\int_{\mathbb{R}^{n}}e^{-i\langle x,u\rangle}h(\Vert u\Vert)g(\Vert u\Vert)du
\end{equation}
is the Fourier transform of $h(\cdot)g(\cdot)$. 
\begin{remark}\label{rem1new}
	Note, that from the isonormal spectral representation~{\rm(\ref{eq2xi})} and the It{\^o} formula
	\begin{align}\label{eq9new}
	H_{\kappa}(\xi(x+y))=\int_{\mathbb{R}^{n\kappa}}^{\prime}e^{i\langle\lambda_{1}+\dots+\lambda_{\kappa},x+y\rangle}\prod_{j=1}^{\kappa}\sqrt{f(\Vert\lambda_{j}\Vert)}\prod_{j=1}^{\kappa}W(d\lambda_{j})
	\end{align}
	it follows that
	\begin{align*}
	&V(x)=\int_{\mathbb{R}^{n}}G(\Vert y\Vert)\int_{\mathbb{R}^{n\kappa}}^{\prime}e^{i\langle\lambda_{1}+\dots+\lambda_{\kappa},x+y\rangle}\prod_{j=1}^{\kappa}\sqrt{f(\Vert\lambda_{j}\Vert)}\prod_{j=1}^{\kappa}W(d\lambda_{j})dy\\
	&=\int_{\mathbb{R}^{n}}e^{i\langle\lambda_{1}+\dots+\lambda_{\kappa},x\rangle}\int_{\mathbb{R}^{n\kappa}}^{\prime}e^{i\langle\lambda_{1}+\dots+\lambda_{\kappa},y\rangle}G(\Vert y\Vert)dy\prod_{j=1}^{\kappa}\left(\sqrt{f(\Vert\lambda_{j}\Vert)}W(d\lambda_{j})\right)\\
	&=\int_{\mathbb{R}^{n\kappa}}^{\prime}e^{i\langle\lambda_{1}+\dots+\lambda_{\kappa},x\rangle}\hat{G}(\lambda_{1}+\dots+\lambda_{\kappa})\prod_{j=1}^{\kappa}\sqrt{f(\Vert\lambda_{j}\Vert)}\prod_{j=1}^{\kappa}W(d\lambda_{j}),
	\end{align*}
	where $\hat{G}(\cdot)$ is the Fourier transform of the function $G(\cdot)$ that is defined by~{\rm(\ref{eq6})} and the stochastic Fubini's theorem~{\rm\cite[Theorem 5.13.1]{peccati2011wiener}} was used to interchange the order of integration. 
	
	By~{\rm(\ref{eq6})} and Assumption~{\rm\ref{ass4}} the isonormal spectral representation of $V(x)$ is
	\begin{align*}
V(x)&=\int_{\mathbb{R}^{n\kappa}}^{\prime}e^{i\langle\lambda_{1}+\dots+\lambda_{\kappa},x\rangle}h(\Vert\lambda_{1}+\dots+\lambda_{\kappa}\Vert)g(\Vert\lambda_{1}+\dots+\lambda_{\kappa}\Vert)\\
&\times\prod_{j=1}^{\kappa}\sqrt{f(\Vert\lambda_{j}\Vert)}\prod_{j=1}^{\kappa}W(d\lambda_{j})=h(1)\int_{\mathbb{R}^{n\kappa}}^{\prime}e^{i\langle\lambda_{1}+\dots+\lambda_{\kappa},x\rangle}\Vert\lambda_{1}+\dots+\lambda_{\kappa}\Vert^{\beta}\\
&\times g(\Vert\lambda_{1}+\dots+\lambda_{\kappa}\Vert)\prod_{j=1}^{\kappa}\sqrt{f(\Vert\lambda_{j}\Vert)}\prod_{j=1}^{\kappa}W(d\lambda_{j}).
\end{align*}
Therefore, it follows that the covariance of $V(x)$ is 
\begin{align}\label{eq10new}
Cov(V(x),V(y))&=h^{2}(1)\int_{\mathbb{R}^{n\kappa}}e^{i\langle\lambda_{1}+\dots+\lambda_{\kappa},x-y\rangle}\Vert\lambda_{1}+\dots+\lambda_{\kappa}\Vert^{2\beta} \notag\\&\times g^{2}(\Vert\lambda_{1}+\dots+\lambda_{\kappa}\Vert)\prod_{j=1}^{\kappa}f(\Vert\lambda_{j}\Vert)d\lambda_{j}.
\end{align}

\end{remark}
\begin{remark}\label{rem1}
By the homogeneity of $h(\cdot)$ and Lemma {\rm3} in~{\rm\cite{leonenko2014sojourn}} it holds
$$\mathcal{I}_{1}(\alpha):=\int_{\mathbb{R}^{n}}|K_{\Delta}\left(\lambda\right)|^{2}\dfrac{h^{2}\left(\Vert\lambda\Vert\right)d\lambda}{\Vert\lambda\Vert^{n-\alpha}}=h^{2}(1)\int_{\mathbb{R}^{n}}|K_{\Delta}\left(\lambda\right)|^{2}\dfrac{d\lambda}{\Vert\lambda\Vert^{n-\alpha-2\beta}}<\infty,$$
for $\alpha\in\left(0,n-2\beta\right)$ and $\beta<n/2$.
\end{remark}

\begin{lemma}\label{lem2}
	If $\tau_{1},\dots,\tau_{\kappa},\ \kappa\geq 1$, are positive constants such that it holds\\
	$\sum_{i=1}^{\kappa}\tau_{i}<n-2\beta$ and $\beta<n/2$, then
	\begin{align*}
	\mathcal{I}_{\kappa}(\tau_{1},\dots,\tau_{\kappa}):&=\int_{\mathbb{R}^{n\kappa}}\big\vert K_{\Delta}\left(\lambda_{1}+\cdots+\lambda_{\kappa}\right)\big\vert^{2}\Vert\lambda_{1}+\cdots+\lambda_{\kappa}\Vert^{2\beta}\\
	&\times\dfrac{\prod_{j=1}^{\kappa} d\lambda_{j}}{\Vert\lambda_{1}\Vert^{n-\tau_{1}}\cdots\Vert\lambda_{\kappa}\Vert^{n-\tau_{\kappa}}}<\infty.
	\end{align*}
\end{lemma}
\begin{proof}
	For $\kappa=1$ we have $\tau_{1}\in(0,n-2\beta)$ and by Remark~\ref{rem1} we get the statement of the Lemma.
	
	For $\kappa>1$, let us use the change of variables $\tilde{\lambda}_{\kappa-1}=\lambda_{\kappa-1}/\Vert u\Vert$, where $u=\lambda_{\kappa}+\lambda_{\kappa-1}$. Then, by applying the recursive estimation routine we get
	\begin{align*}
	&\mathcal{I}_{\kappa}(\tau_{1},\dots,\tau_{\kappa})=\int_{\mathbb{R}^{n(\kappa-1)}}|K_{\Delta}\left(\lambda_{1}+\cdots+\lambda_{\kappa-2}+u\right)|^{2}\\
	&\times\int_{\mathbb{R}^{n}}\dfrac{\Vert\lambda_{1}+\cdots+\lambda_{\kappa-2}+u\Vert^{2\beta} d\lambda_{\kappa-1}}{\Vert\lambda_{\kappa-1}\Vert^{n-\tau_{\kappa-1}}\Vert u-\lambda_{\kappa-1}\Vert^{n-\tau_{\kappa}}}\dfrac{d\lambda_{1}\cdots d\lambda_{\kappa-2}du}{\Vert\lambda_{1}\Vert^{n-\tau_{1}}\cdots
		\Vert\lambda_{\kappa-2}\Vert^{n-\tau_{\kappa-2}}}
	\end{align*}  
	\begin{align}\label{eq7}
	&=\int\limits_{\mathbb{R}^{n}}\bigg( \int\limits_{\mathbb{R}^{n(\kappa-2)}}\dfrac{|K_{\Delta}\left(\lambda_{1}+\cdots+\lambda_{\kappa-2}+u\right)|^{2}\Vert\lambda_{1}+\cdots+\lambda_{\kappa-2}+u\Vert^{2\beta}\prod_{j=1}^{\kappa-2}d\lambda_{j}}{\Vert\lambda_{1}\Vert^{n-\tau_{1}}\cdots\Vert\lambda_{\kappa-2}\Vert^{n-\tau_{\kappa-2}}\Vert u\Vert^{n-\tau_{\kappa-1}-\tau_{\kappa}}}\notag\\
	&\times\int_{\mathbb{R}^{n}}\dfrac{d\tilde{\lambda}_{\kappa-1}}{\Vert\tilde{\lambda}_{\kappa-1}\Vert^{n-\tau_{\kappa-1}}\big\Vert \frac{u}{\Vert u\Vert}-\tilde{\lambda}_{\kappa-1}\big\Vert^{n-\tau_{\kappa}}}\bigg)du.
	\end{align}
	
	Note, that the second integrand in~(\ref{eq7}) is unbounded at $\Vert\tilde{\lambda}_{\kappa-1}\Vert=0$ and $\tilde{\lambda}_{\kappa-1}=u/\Vert u\Vert$ (in this case $\Vert\tilde{\lambda}_{\kappa-1}\Vert=1$). If we split $\mathbb{R}^{n}$ into the regions
	$
	A_{1}:=\{\tilde{\lambda}_{\kappa-1}\in\mathbb{R}^{n}:\Vert\tilde{\lambda}_{\kappa-1}\Vert<\frac{1}{2}\}$, $
	A_{2}:=\{\tilde{\lambda}_{\kappa-1}\in\mathbb{R}^{n}:\frac{1}{2}\leq\Vert\tilde{\lambda}_{\kappa-1}\Vert<\frac{3}{2}\}$, and $
	A_{3}:=\{\tilde{\lambda}_{\kappa-1}\in\mathbb{R}^{n}:\Vert\tilde{\lambda}_{\kappa-1}\Vert\geq\frac{3}{2}\}$ we get
	\begin{align*}
	&\int_{\mathbb{R}^{n}}\dfrac{d\tilde{\lambda}_{\kappa-1}}{\Vert\tilde{\lambda}_{\kappa-1}\Vert^{n-\tau_{\kappa-1}}\big\Vert \frac{u}{\Vert u\Vert}-\tilde{\lambda}_{\kappa-1}\big\Vert^{n-\tau_{\kappa}}}\leq\sup_{\tilde{\lambda}_{\kappa-1}\in A_{1}}\big\Vert \frac{u}{\Vert u\Vert}-\tilde{\lambda}_{\kappa-1}\big\Vert^{\tau_{\kappa}-n}\\
	&\times\int_{A_{1}}\Vert\tilde{\lambda}_{\kappa-1}\Vert^{\tau_{\kappa-1}-n}d\tilde{\lambda}_{\kappa-1}+\sup_{\tilde{\lambda}_{\kappa-1}\in A_{2}}\Vert\tilde{\lambda}_{\kappa-1}\Vert^{\tau_{\kappa-1}-n}\int_{A_{2}}\dfrac{d\tilde{\lambda}_{\kappa-1}}{\big\Vert\frac{u}{\Vert u\Vert}-\tilde{\lambda}_{\kappa-1}\big\Vert^{n-\tau_{\kappa}}}\\
	&+\int_{A_{3}}\big\Vert\tilde{\lambda}_{\kappa-1}\big\Vert^{\tau_{\kappa-1}-n}\vert\big\Vert\tilde{\lambda}_{\kappa-1}\big\Vert-1\vert^{\tau_{\kappa}-n}d\tilde{\lambda}_{\kappa-1}\leq\left(\dfrac{1}{2}\right)^{\tau_{\kappa}-n}\int_{0}^{1/2}\rho^{\tau_{\kappa-1}-1} d\rho\\
	&+\left(\dfrac{1}{2}\right)^{\tau_{\kappa-1}-n}\int_{A_{2}-\dfrac{u}{\Vert u\Vert}}\Vert\hat{\lambda}_{\kappa-1}\Vert^{\tau_{\kappa}-n}d\hat{\lambda}_{\kappa-1}+\int_{3/2}^{\infty}\rho^{\tau_{\kappa-1}-1}\left(\rho-1\right)^{\tau_{\kappa}-n}d\rho\\
	&\leq C+\left(\dfrac{1}{2}\right)^{\tau_{\kappa-1}-n}\int_{0}^{5/2}\rho^{\tau_{\kappa}-1} d\rho+\int_{1/2}^{\infty}\dfrac{d\hat{\rho}}{\hat{\rho}^{n+1-\tau_{\kappa}-\tau_{\kappa-1}}}=C<\infty,	
	\end{align*}
	where $A_{2}-\dfrac{u}{\Vert u\Vert}=\{\lambda\in\mathbb{R}^{n}:\lambda+\dfrac{u}{\Vert u\Vert}\in A_{2}\}\subset v_{n}\left(\frac{5}{2}\right)$, $v_{n}\left(r\right)$ is a $n$-dimensional ball with center $0$ and radius $r$. 
	
	Hence, by~(\ref{eq7}) and Remark~\ref{rem1}
	\begin{align}\label{eq*}
	\mathcal{I}_{\kappa}(\tau_{1},\dots,\tau_{\kappa})&\leq C\mathcal{I}_{\kappa-1}(\tau_{1},\dots,\tau_{\kappa-2},\tau_{\kappa-1}+\tau_{\kappa})\notag\\
	&\leq\cdots\leq C\mathcal{I}_{1}\left(\sum_{i=1}^{\kappa}\tau_{i}\right)\leq
	C\int_{\mathbb{R}^{n}}\dfrac{|K_{\Delta}\left( u\right)|^{2}du}{\Vert u\Vert^{n-\sum_{i=1}^{\kappa}\tau_{i}-2\beta}}<\infty,
	\end{align}
	which completes the proof.\qed
\end{proof}
\begin{lemma}\label{lem3}
	The following integral is finite
	$$\mathcal{J}_{\kappa}:=\int_{\mathbb{R}^{n\kappa}}\big\vert\hat{G}(\lambda_{1}+\dots+\lambda_{\kappa})\big\vert^{2}\prod_{i=1}^{\kappa}f(\Vert\lambda_{i}\Vert)d\lambda_{i}<\infty.$$
\end{lemma}
\begin{proof}
	As $f(\cdot)$ is an isotropic spectral density we can rewrite $\mathcal{J}_{\kappa}$ as
	\begin{align}\label{eq7b}
	\mathcal{J}_{\kappa}&=\int_{\mathbb{R}^{n(\kappa-1)}}\int_{\mathbb{R}^{n}}\big\vert\hat{G}((\lambda_{1}+\dots+\lambda_{\kappa-1})+\lambda_{\kappa})\big\vert^{2}f(\Vert-\lambda_{\kappa}\Vert)d\lambda_{\kappa}\prod_{i=1}^{\kappa-1}f(\Vert\lambda_{i}\Vert)d\lambda_{i}\notag\\
	&=\int_{\mathbb{R}^{n(\kappa-1)}}\int_{\mathbb{R}^{n}}\left(\big\vert\hat{G}\big\vert^{2}*f\right)(\lambda_{1}+\dots+\lambda_{\kappa-1})\prod_{i=1}^{\kappa-1}f(\Vert-\lambda_{i}\Vert)d\lambda_{i}.
	\end{align}
	Note that $\vert \hat{G}\vert^{2}(\cdot)\in L_{1}\left(\mathbb{R}^{n}\right)$ and $f(\cdot)\in L_{1}\left(\mathbb{R}^{n}\right)$. Hence, by Young's theorem~\cite{bogachev2007measure} it follows that $ \hat{G}_{1}^{2}(\cdot)=\left(\big\vert\hat{G}\big\vert^{2}*f\right)(\cdot)\in L_{1}\left(\mathbb{R}^{n}\right)$. Therefore, using convolutions as in~(\ref{eq7b}) we obtain
	\begin{align*}
	\mathcal{J}_{\kappa}&=\int_{\mathbb{R}^{n(\kappa-2)}}\int_{\mathbb{R}^{n}}\hat{G}_{1}^{2}(\lambda_{1}+\dots+\lambda_{\kappa-1})f(\Vert-\lambda_{\kappa-1}\Vert)d\lambda_{\kappa-1}\prod_{i=1}^{\kappa-2}f(\Vert-\lambda_{i}\Vert)d\lambda_{i}\notag\\
	&=\int_{\mathbb{R}^{n(\kappa-2)}}\int_{\mathbb{R}^{n}}\left(\hat{G}_{1}^{2}*f\right)(\lambda_{1}+\dots+\lambda_{\kappa-2})\prod_{i=1}^{\kappa-2}f(\Vert-\lambda_{i}\Vert)d\lambda_{i}\\
	&=\int_{\mathbb{R}^{n(\kappa-2)}}\int_{\mathbb{R}^{n}}\hat{G}_{2}^{2}(\lambda_{1}+\dots+\lambda_{\kappa-2})\prod_{i=1}^{\kappa-2}f(\Vert-\lambda_{i}\Vert)d\lambda_{i}=\cdots=\\
	&=\int_{\mathbb{R}^{n}}\hat{G}_{\kappa-1}^{2}(\lambda_{1})f(\Vert-\lambda_{1}\Vert)d\lambda_{1}<\infty,
	\end{align*}
	where $ \hat{G}_{j+1}^{2}(\cdot):=\left(\hat{G}_{j}^{2}*f\right)(\cdot)\in L_{1}\left(\mathbb{R}^{n}\right)$ by Young's theorem and recursive steps.\qed 
\end{proof}

Now we proceed to the main result.
\begin{theorem}\label{theo4}
	Let $\xi\left(x\right),\  x\in\mathbb{R}^{n}$, be a random field satisfying Assumptions~{\rm\ref{ass1}}, {\rm\ref{ass2}} and functions $g(\cdot)$ and $h(\cdot)$ satisfy Assumption~{\rm\ref{ass4}}. Then, for $r \rightarrow +\infty$ the finite-dimensional distributions of 
	$$X_{r,\kappa}(t):=\dfrac{r^{\beta+\kappa\alpha/2-n}\mathcal{L}^{-\kappa/2}(r)}{(2\pi)^{n}c_{1}^{\kappa/2}(n,\alpha)g(0)h(1)}\int_{\Delta(rt^{1/n})}V\left(x\right)dx,\quad t\in[0,1],$$ converge weakly to the finite-dimensional distributions of 
	$$X_{\kappa}(t):=t\int_{\mathbb{R}^{n\kappa}}^{\prime}K_{\Delta}\left(\left(\lambda_{1}+\dots+\lambda_{\kappa}\right) t^{1/n}\right)\dfrac{\Vert \lambda_{1}+\dots+\lambda_{\kappa}\Vert^{\beta}\prod_{j=1}^{\kappa}W(d\lambda_{j})}{\prod_{j=1}^{\kappa}\Vert\lambda_{j}\Vert^{(n-\alpha)/2}},$$\ 
	where
	$\alpha\in\left(0,\frac{n-2\beta}{\kappa}\right)$ and $\beta<\frac{n}{2}$.
\end{theorem}
\begin{remark}\label{rem2new}
	By the representation~{\rm(\ref{eq10new})} in Remark~{\rm\ref{rem1new}} of the covariance function of $V(x)$ we obtain
	\begin{align*}
	&Cov\left(X_{r,\kappa}(t),X_{r,\kappa}(s)\right)=\\
	&=\left(\dfrac{r^{\beta+\kappa\alpha/2-n}\mathcal{L}^{-\kappa/2}(r)}{(2\pi)^{n}c_{1}^{\kappa/2}(n,\alpha)g(0)h(1)}\right)^{2}\int_{\Delta(rt^{1/n})}\int_{\Delta(rs^{1/n})}Cov\left(V(x),V(y)\right)dxdy\\
	&=\left(\dfrac{r^{\beta+\kappa\alpha/2-n}\mathcal{L}^{-\kappa/2}(r)}{(2\pi)^{n}c_{1}^{\kappa/2}(n,\alpha)g(0)}\right)^{2}\int_{\Delta(rt^{1/n})}\int_{\Delta(rs^{1/n})}\int_{\mathbb{R}^{n\kappa}}e^{i\langle\lambda_{1}+\dots+\lambda_{\kappa},x-y\rangle}
	\\
	&\times \Vert\lambda_{1}+\dots+\lambda_{\kappa}\Vert^{2\beta}g^{2}(\Vert\lambda_{1}+\dots+\lambda_{\kappa}\Vert)\prod_{j=1}^{\kappa}f(\Vert\lambda_{j}\Vert)d\lambda_{j}dxdy\\
	&=\left(\dfrac{r^{\beta+\kappa\alpha/2-n}\mathcal{L}^{-\kappa/2}(r)}{(2\pi)^{n}c_{1}^{\kappa/2}(n,\alpha)g(0)}\right)^{2}\int_{\mathbb{R}^{n\kappa}}	\int_{\Delta(rt^{1/n})}\int_{\Delta(rs^{1/n})}e^{i\langle\lambda_{1}+\dots+\lambda_{\kappa},x-y\rangle}dxdy
	\\
	&\times \Vert\lambda_{1}+\dots+\lambda_{\kappa}\Vert^{2\beta}g^{2}(\Vert\lambda_{1}+\dots+\lambda_{\kappa}\Vert)\prod_{j=1}^{\kappa}f(\Vert\lambda_{j}\Vert)d\lambda_{j}\\
	&=\left(\dfrac{r^{\beta+\kappa\alpha/2}\mathcal{L}^{-\kappa/2}(r)}{(2\pi)^{n}c_{1}^{\kappa/2}(n,\alpha)g(0)}\right)^{2} ts\int_{\mathbb{R}^{n\kappa}}K_{\Delta}\left(\left(\lambda_{1}+\dots+\lambda_{\kappa}\right)rt^{1/n}\right)\prod_{j=1}^{\kappa}f(\Vert\lambda_{j}\Vert)\\
	&\times \overline{K_{\Delta}\left(\left(\lambda_{1}+\dots+\lambda_{\kappa}\right)rs^{1/n}\right)}\Vert\lambda_{1}+\dots+\lambda_{\kappa}\Vert^{2\beta}g^{2}(\Vert\lambda_{1}+\dots+\lambda_{\kappa}\Vert)\prod_{j=1}^{\kappa}d\lambda_{j}.
	\end{align*}
	In particular, the variance of $X_{r,\kappa}(t)$ is equal
	\begin{align*}
	&Var\left(X_{r,\kappa}(t)\right)=\left(\dfrac{r^{\beta+\kappa\alpha/2}\mathcal{L}^{-\kappa/2}(r)}{(2\pi)^{n}c_{1}^{\kappa/2}(n,\alpha)g(0)}\right)^{2}t^{2}\int_{\mathbb{R}^{n\kappa}}\Vert\lambda_{1}+\dots+\lambda_{\kappa}\Vert^{2\beta}\\
	&\times\bigg\vert K_{\Delta}\left(\left(\lambda_{1}+\dots+\lambda_{\kappa}\right)rt^{1/n}\right)\bigg\vert^{2}g^{2}(\Vert\lambda_{1}+\dots+\lambda_{\kappa}\Vert)\prod_{j=1}^{\kappa}f(\Vert\lambda_{j}\Vert)d\lambda_{j}.
	\end{align*}
	Similarly, we get
	\begin{align*}
	&Cov\left(X_{\kappa}(t),X_{\kappa}(s)\right)=ts\int_{\mathbb{R}^{n\kappa}}K_{\Delta}\left(\left(\lambda_{1}+\dots+\lambda_{\kappa}\right)t^{1/n}\right)\\
	&\times\overline{K_{\Delta}\left(\left(\lambda_{1}+\dots+\lambda_{\kappa}\right)s^{1/n}\right)}\Vert\lambda_{1}+\dots+\lambda_{\kappa}\Vert^{2\beta}\prod_{j=1}^{\kappa}\Vert\lambda_{j}\Vert^{\alpha-n}d\lambda_{j}
	\end{align*}
	and
	\begin{align*}
	Var\left(X_{\kappa}(t)\right)&=t^{2}\int_{\mathbb{R}^{n\kappa}}\big\vert K_{\Delta}\left(\left(\lambda_{1}+\dots+\lambda_{\kappa}\right)t^{1/n}\right)\big\vert^{2}\dfrac{\Vert\lambda_{1}+\dots+\lambda_{\kappa}\Vert^{2\beta}\prod_{j=1}^{\kappa}d\lambda_{j}}{\prod_{j=1}^{\kappa}\Vert\lambda_{j}\Vert^{n-\alpha}}.
	\end{align*}	
\end{remark}

\begin{remark}\label{rem2}
	Note, that for $a>0$ we have
	\begin{align*}
	X_{\kappa}(at)&=at\int_{\mathbb{R}^{n\kappa}}^{\prime}K_{\Delta}\left(\left(\lambda_{1}+\dots+\lambda_{\kappa}\right) (at)^{1/n}\right)\dfrac{\Vert \lambda_{1}+\dots+\lambda_{\kappa}\Vert^{\beta}\prod_{j=1}^{\kappa}W(d\lambda_{j})}{\prod_{j=1}^{\kappa}\Vert\lambda_{j}\Vert^{\frac{n-\alpha}{2}}}.
	\end{align*}
	Using the transformation $\tilde{\lambda}_{j}=a^{\frac{1}{n}}\lambda_{j},\ j=1,\dots,\kappa$, and the self-similarity of the Gaussian white noise we get
	\begin{align*}
	X_{\kappa}(at)&=\dfrac{a^{1-\frac{\beta}{n}}}{\left(a^{-\frac{1}{n}}\right)^{\kappa(n-\alpha)/2}}t\int_{\mathbb{R}^{n\kappa}}^{\prime}K_{\Delta}\left(\left(\tilde{\lambda}_{1}+\dots+\tilde{\lambda}_{\kappa}\right) t^{1/n}\right)\dfrac{\Vert \tilde{\lambda}_{1}+\dots+\tilde{\lambda}_{\kappa}\Vert^{\beta}}{\prod_{j=1}^{\kappa}\Vert\tilde{\lambda}_{j}\Vert^{(n-\alpha)/2}}\\
	&\times\prod_{j=1}^{\kappa}W(a^{-\frac{1}{n}}d\tilde{\lambda}_{j})=a^{1-\frac{\kappa\alpha}{2n}-\frac{\beta}{n}}X_{\kappa}(t).
	\end{align*}	
	
	Thus, the random process $X_{\kappa}(t)$ is a self-similar with the Hurst parameter~$H=1-\frac{\kappa\alpha}{2n}-\frac{\beta}{n}$.
\end{remark}

\begin{proof}
	By (\ref{eq5}) the process $X_{r,\kappa}(t)$ admits the following representation
	\[
	X_{r,\kappa}(t)=\dfrac{r^{\beta+\kappa\alpha/2-n}\mathcal{L}^{-\kappa/2}(r)}{(2\pi)^{n}c_{1}^{\kappa/2}(n,\alpha)g(0)h(1)}\int_{\Delta(rt^{1/n})}\left(\int_{\mathbb{R}^{n}}G(\Vert y\Vert)H_{\kappa}(\xi(x+y))dy\right)dx.
	\]
	By~(\ref{eq9new}) we obtain
	\begin{align}\label{eq13}
	X_{r,\kappa}(t)&=\dfrac{r^{\beta+\kappa\alpha/2-n}\mathcal{L}^{-\kappa/2}(r)}{(2\pi)^{n}c_{1}^{\kappa/2}(n,\alpha)g(0)h(1)}\int_{\Delta(rt^{1/n})}\bigg(\int_{\mathbb{R}^{n}}G(\Vert y\Vert)\notag\\
	&\times\bigg[\int_{\mathbb{R}^{n\kappa}}^{\prime}e^{i\langle\lambda_{1}+\dots+\lambda_{\kappa},x+y\rangle}\prod_{j=1}^{\kappa}\sqrt{f(\Vert\lambda_{j}\Vert)}\prod_{j=1}^{\kappa}W(d\lambda_{j})\bigg]dy\bigg)dx.
	\end{align}
	
	By Assumption~\ref{ass2} it follows $\prod_{j=1}^{\kappa}\sqrt{f(\Vert\lambda_{j}\Vert)}\in L_{2}\left(\mathbb{R}^{n\kappa}\right)$. By Assumption~{\rm\ref{ass4}}, ({\rm\ref{eq6}}) and Parseval's theorem $ G(\cdot)\in L_{2}(\mathbb{R}^{n})$. So, one can apply the stochastic Fubini's theorem to interchange the inner integrals in~(\ref{eq13}), see Theorem 5.13.1 in~\cite{peccati2011wiener}, which results in
	\begin{align}\label{eq8}
	X_{r,\kappa}(t)&=\dfrac{r^{\beta+\kappa\alpha/2-n}\mathcal{L}^{-\kappa/2}(r)}{c_{1}^{\kappa/2}(n,\alpha)g(0)h(1)}\int_{\Delta(rt^{1/n})}\int_{\mathbb{R}^{n\kappa}}^{\prime}e^{i\langle\lambda_{1}+\dots+\lambda_{\kappa},x\rangle}\hat{G}(\lambda_{1}+\dots+\lambda_{\kappa})\notag\\
	&\times\prod_{j=1}^{\kappa}\sqrt{f(\Vert\lambda_{j}\Vert)}\prod_{j=1}^{\kappa}W(d\lambda_{j})dx.
	\end{align}
	
	Note, that by Lemma~\ref{lem3} the integrand in~(\ref{eq8}) belongs to $L_{2}(\mathbb{R}^{n\kappa})$. Then, it follows from the stochastic Fubini's theorem, and Assumption~\ref{ass4} that
	\begin{align*}
	X_{r,\kappa}(t)&=\dfrac{r^{\beta+\kappa\alpha/2-n}\mathcal{L}^{-\kappa/2}(r)}{c_{1}^{\kappa/2}(n,\alpha)g(0)}\int_{\mathbb{R}^{n\kappa}}^{\prime}K_{\Delta(rt^{1/n})}(\lambda_{1}+\dots+\lambda_{\kappa})\\
	&\times\Vert\lambda_{1}+\dots+\lambda_{\kappa}\Vert^{\beta} g(\Vert\lambda_{1}+\dots+\lambda_{\kappa}\Vert)\prod_{j=1}^{\kappa}\sqrt{f(\Vert\lambda_{j}\Vert)}\prod_{j=1}^{\kappa}W(d\lambda_{j}),
	\end{align*}
	where
	\[
	K_{\Delta(rt^{1/n})}(\lambda)=\int_{\Delta(rt^{1/n})}e^{i\langle\lambda,x\rangle}dx.
	\]
	
	Note, that $K_{\Delta(rt^{1/n})}(\lambda)=tr^{n}K_{\Delta}\left( \lambda rt^{1/n}\right)$, where $K_{\Delta}(\cdot)$ is given by (\ref{eq3}). Therefore,
	\begin{align*}
	X_{r,\kappa}(t)&=t\dfrac{r^{\beta+\kappa\alpha/2}\mathcal{L}^{-\kappa/2}(r)}{c_{1}^{\kappa/2}(n,\alpha)g(0)}\int_{\mathbb{R}^{n\kappa}}^{\prime}K_{\Delta}\left(\left(\lambda_{1}+\dots+\lambda_{\kappa}\right)rt^{1/n}\right)\\
	&\times\Vert\lambda_{1}+\dots+\lambda_{\kappa}\Vert^{\beta} g(\Vert\lambda_{1}+\dots+\lambda_{\kappa}\Vert)\prod_{j=1}^{\kappa}\sqrt{f(\Vert\lambda_{j}\Vert)}\prod_{j=1}^{\kappa}W(d\lambda_{j}).
	\end{align*}
	Using the transformation $\lambda^{(j)}=r\lambda_{j},\ j=1,\dots,\kappa$, and the self-similarity of the Gaussian white noise we get
	\begin{align*}
	X_{r,\kappa}(t)&=t \dfrac{r^{\beta+ \kappa\alpha/2}\mathcal{L}^{-\kappa/2}(r)}{c_{1}^{\kappa/2}(n,\alpha)g(0)}
	\int_{\mathbb{R}^{n\kappa}}^{\prime}K_{\Delta}\left(\left(\lambda^{(1)}+\dots+\lambda^{(\kappa)}\right)t^{1/n}\right)\\
	&\times\left(r^{-1}\big\Vert\lambda^{(1)}+\dots+\lambda^{(\kappa)}\big\Vert\right)^{\beta}\prod_{j=1}^{\kappa}\sqrt{f(\Vert\lambda^{(j)}\Vert/r)}\\
	&\times g\left(r^{-1}\big\Vert\lambda^{(1)}+\dots+\lambda^{(\kappa)}\big\Vert\right)\prod_{j=1}^{\kappa}W(d\lambda^{(j)}/r)=t \dfrac{r^{\kappa\alpha/2}\mathcal{L}^{-\kappa/2}(r)r^{-n\kappa/2}}{c_{1}^{\kappa/2}(n,\alpha)g(0)}\\
	&\times\int_{\mathbb{R}^{n\kappa}}^{\prime}K_{\Delta}\left(\left(\lambda^{(1)}+\dots+\lambda^{(\kappa)}\right)t^{1/n}\right)\big\Vert\lambda^{(1)}+\dots+\lambda^{(\kappa)}\big\Vert^{\beta}\\
	&\times\prod_{j=1}^{\kappa}\sqrt{f(\Vert\lambda^{(j)}\Vert/r)}g\left(r^{-1}\big\Vert\lambda^{(1)}+\dots+\lambda^{(\kappa)}\big\Vert\right)\prod_{j=1}^{\kappa}W(d\lambda^{(j)})\\
	&=t\dfrac{r^{\kappa(\alpha-n)/2}\mathcal{L}^{-\kappa/2}(r)}{g(0)}\int_{\mathbb{R}^{n\kappa}}^{\prime}K_{\Delta}\left(\left(\lambda^{(1)}+\dots+\lambda^{(\kappa)}\right)t^{1/n}\right)\\
	&\times
	\big\Vert\lambda^{(1)}+\dots+\lambda^{(\kappa)}\big\Vert^{\beta}\prod_{j=1}^{\kappa}\sqrt{(\Vert\lambda^{(j)}\Vert/r)^{\alpha-n}\mathcal{L}(r/\Vert\lambda^{(j)}\Vert)}\\
	&\times g\left(r^{-1}\big\Vert\lambda^{(1)}+\dots+\lambda^{(\kappa)}\big\Vert\right)\prod_{j=1}^{\kappa}W(d\lambda^{(j)})\\
	&=\dfrac{t}{g(0)}\int_{\mathbb{R}^{n\kappa}}^{\prime}\dfrac{K_{\Delta}\left(\left(\lambda^{(1)}+\dots+\lambda^{(\kappa)}\right)t^{1/n}\right)}{\prod_{j=1}^{\kappa}\Vert\lambda^{(j)}\Vert^{(n-\alpha)/2}}
	\big\Vert\lambda^{(1)}+\dots+\lambda^{(\kappa)}\big\Vert^{\beta}\\
	&\times\prod_{j=1}^{\kappa}\sqrt{\mathcal{L}(r/\Vert\lambda^{(j)}\Vert)/\mathcal{L}(r)}g\left(r^{-1}\big\Vert\lambda^{(1)}+\dots+\lambda^{(\kappa)}\big\Vert\right)\prod_{j=1}^{\kappa}W(d\lambda^{(j)}).
	\end{align*}
	
	By the isometry property of multiple stochastic integrals
	\begin{align*}
	R_{r}:&=E\left(X_{r,\kappa}(t)-X_{\kappa}(t)\right)^{2}=t^{2}\int_{\mathbb{R}^{n\kappa}}\dfrac{\vert K_{\Delta}\left(\left(\lambda^{(1)}+\dots+\lambda^{(\kappa)}\right)t^{1/n}\right)\vert^{2}}{\prod_{j=1}^{\kappa}\Vert\lambda^{(j)}\Vert^{n-\alpha}}\\
	&\times \big\Vert\lambda^{(1)}+\dots+\lambda^{(\kappa)}\big\Vert^{2\beta}\left(Q_{r}(\lambda^{(1)},\dots,\lambda^{(\kappa)})-1\right)^{2}d\lambda^{(1)}\cdots d\lambda^{(\kappa)},
	\end{align*}
	where
	\[
	Q_{r}\left(\lambda^{(1)},\dots,\lambda^{(\kappa)}\right):=\dfrac{g\left(r^{-1}\big\Vert\lambda^{(1)}+\dots+\lambda^{(\kappa)}\big\Vert\right)}{g(0)}\sqrt{\prod_{j=1}^{\kappa}\mathcal{L}(r/\Vert\lambda^{(j)}\Vert)/\mathcal{L}(r)}.
	\]
	Note, that by Assumptions~\ref{ass2}, \ref{ass4}, and properties of slowly varying functions $Q_{r}(\lambda^{(1)},\dots,\lambda^{(\kappa)})$ pointwise converges to 1, when $r \rightarrow\infty$.
	%Hence, by Lebsque's dominated convergence theorem the integral converges to zero if there is some integrable function which dominates integrands for all $r>0$.
	
	Let us split $\mathbb{R}^{n\kappa}$ into the regions
	\begin{align*}
	B_{\mu}:=\bigg\{\left(\lambda^{(1)},\dots,\lambda^{(\kappa)}\right)\in\mathbb{R}^{n\kappa}&:\Vert\lambda^{(j)}\Vert\leq1,\ \text{if}\ \mu_{j}=-1,\\ &\text{and}\ \Vert\lambda^{(j)}\Vert>1,\ \text{if}\ \mu_{j}=1,j=1,\dots,\kappa\bigg\},
	\end{align*} 
	where $\mu=\left(\mu_{1},\dots,\mu_{\kappa}\right)\in\{-1,1\}^{\kappa}$ is a binary vector of length $\kappa$. Then we can represent the integral $R_{r}$ as
	\begin{align*}
	R_{r}&=t^{2}\int_{\cup_{\mu\in\{-1,1\}^{\kappa}}B_{\mu}}\dfrac{\vert K_{\Delta}\left(\left(\lambda^{(1)}+\dots+\lambda^{(\kappa)}\right)t^{1/n}\right)\vert^{2}}{\prod_{j=1}^{\kappa}\Vert\lambda^{(j)}\Vert^{n-\alpha}}\\
	&\times\big\Vert\lambda^{(1)}+\dots+\lambda^{(\kappa)}\big\Vert^{2\beta}\left(Q_{r}\left(\lambda^{(1)},\dots,\lambda^{(\kappa)}\right)-1\right)^{2}d\lambda^{(1)}\cdots d\lambda^{(\kappa)}.
	\end{align*}
	
	If $\left(\lambda^{(1)},\dots,\lambda^{(\kappa)}\right)\in B_{\mu}$ we estimate the integrand as follows
	\[
	\dfrac{\vert K_{\Delta}\left(\left(\lambda^{(1)}+\dots+\lambda^{(\kappa)}\right)t^{1/n}\right)\vert^{2}}{\prod_{j=1}^{\kappa}\Vert\lambda^{(j)}\Vert^{n-\alpha}}\big\Vert\lambda^{(1)}+\dots+\lambda^{(\kappa)}\big\Vert^{2\beta}\left(Q_{r}(\lambda^{(1)},\dots,\lambda^{(\kappa)})-1\right)^{2}
	\]
	\[
	\leq \dfrac{2\vert K_{\Delta}\left(\left(\lambda^{(1)}+\dots+\lambda^{(\kappa)}\right)t^{1/n}\right)\vert^{2}}{\prod_{j=1}^{\kappa}\Vert\lambda^{(j)}\Vert^{n-\alpha}}\big\Vert\lambda^{(1)}+\dots+\lambda^{(\kappa)}\big\Vert^{2\beta}\left(Q_{r}^{2}(\lambda^{(1)},\dots,\lambda^{(\kappa)})+1\right)
	\]
	\[
	=\dfrac{2\vert K_{\Delta}\left(\left(\lambda^{(1)}+\dots+\lambda^{(\kappa)}\right)t^{1/n}\right)\vert^{2}}{\prod_{j=1}^{\kappa}\Vert\lambda^{(j)}\Vert^{n-\alpha}}\big\Vert\lambda^{(1)}+\dots+\lambda^{(\kappa)}\big\Vert^{2\beta}
	\]
	\[
	\times\left(1+\dfrac{g^{2}\left(r^{-1}\big\Vert\lambda^{(1)}+\dots+\lambda^{(\kappa)}\big\Vert\right)}{g^{2}(0)}\prod_{j=1}^{\kappa}\Vert\lambda^{(j)}\Vert^{\mu_{j}\delta}\prod_{j=1}^{\kappa}\dfrac{(\frac{r}{\Vert\lambda^{(j)}\Vert})^{\mu_{j}\delta}\mathcal{L}(\frac{r}{\Vert\lambda^{(j)}\Vert})}{r^{\mu_{j}\delta}\mathcal{L}(r)}\right),
	\]
	where $\delta$ is an arbitrary positive number.
	
	Using the boundedness of the function $ g(\cdot) $, we can write
	\[\dfrac{\vert K_{\Delta}\left(\left(\lambda^{(1)}+\dots+\lambda^{(\kappa)}\right)t^{1/n}\right)\vert^{2}}{\prod_{j=1}^{\kappa}\Vert\lambda^{(j)}\Vert^{n-\alpha}}\big\Vert\lambda^{(1)}+\dots+\lambda^{(\kappa)}\big\Vert^{2\beta}\left(Q_{r}(\lambda^{(1)},\dots,\lambda^{(\kappa)})-1\right)^{2}\]
	\[\leq \dfrac{2\bigg\vert K_{\Delta}\left(\left(\lambda^{(1)}+\dots+\lambda^{(\kappa)}\right)t^{1/n}\right)\bigg\vert^{2}}{\prod_{j=1}^{\kappa}\Vert\lambda^{(j)}\Vert^{n-\alpha}}\big\Vert\lambda^{(1)}+\dots+\lambda^{(\kappa)}\big\Vert^{2\beta}\]
	\[\times\left(1+C\sup_{\left(\lambda_{1},\dots,\lambda_{\kappa}\right)\in
		B_{\mu}}\prod_{j=1}^{\kappa}\Vert\lambda^{(j)}\Vert^{\mu_{j}\delta}\prod_{j=1}^{\kappa}\dfrac{(r/\Vert\lambda^{(j)}\Vert)^{\mu_{j}\delta}\mathcal{L}(r/\Vert\lambda^{(j)}\Vert)}{r^{\mu_{j}\delta}\mathcal{L}(r)}\right).
	\] 

	By Theorem~\ref{the1new}
	\begin{align*}
	\lim_{r\rightarrow\infty}\dfrac{\sup_{\Vert\lambda^{(j)}\Vert\leq1}\left(r/\Vert\lambda^{(j)}\Vert\right)^{-\delta}\mathcal{L}\left(r/\Vert\lambda^{(j)}\Vert\right)}{r^{-\delta}\mathcal{L}(r)}=1;
	\end{align*}
	and
	\begin{align*}
	\lim_{r\rightarrow\infty}\dfrac{\sup_{\Vert\lambda^{(j)}\Vert>1}\left(r/\Vert\lambda^{(j)}\Vert\right)^{\delta}\mathcal{L}\left(r/\Vert\lambda^{(j)}\Vert\right)}{r^{\delta}\mathcal{L}(r)}=1.
	\end{align*}
	Therefore, there exists $r_{0}>0$ such that for all $r\geq r_{0}$ and $\left(\lambda^{(1)},\dots,\lambda^{(\kappa)}\right)\in B_{\mu}$
	\[\dfrac{\bigg\vert K_{\Delta}\left(\left(\lambda^{(1)}+\dots+\lambda^{(\kappa)}\right)t^{1/n}\right)\bigg\vert^{2}}{\prod_{j=1}^{\kappa}\Vert\lambda^{(j)}\Vert^{n-\alpha}}\big\Vert\lambda^{(1)}+\dots+\lambda^{(\kappa)}\big\Vert^{2\beta}\left(Q_{r}(\lambda^{(1)},\dots,\lambda^{(\kappa)})-1\right)^{2}\]
	\[\leq \dfrac{2\bigg\vert K_{\Delta}\left(\left(\lambda^{(1)}+\dots+\lambda^{(\kappa)}\right)t^{1/n}\right)\bigg\vert^{2}}{\prod_{j=1}^{\kappa}\Vert\lambda^{(j)}\Vert^{n-\alpha}}\big\Vert\lambda^{(1)}+\dots+\lambda^{(\kappa)}\big\Vert^{2\beta}\]
	\begin{equation}\label{eq9}+\dfrac{2C\vert K_{\Delta}\left(\left(\lambda^{(1)}+\dots+\lambda^{(\kappa)}\right)t^{1/n}\right)\vert^{2}}{\prod_{j=1}^{\kappa}\Vert\lambda^{(j)}\Vert^{n-\alpha-\mu_{j}\delta}}\big\Vert\lambda^{(1)}+\dots+\lambda^{(\kappa)}\big\Vert^{2\beta}.
	\end{equation}
	
	By Lemma \ref{lem2}, if we choose $\delta\in\left(0,\min\{\alpha,\frac{n-2\beta}{\kappa}-\alpha\}\right)$, the upper bound in (\ref{eq9}) is an integrable function on each $B_{\mu}$ and hence on $\mathbb{R}^{n\kappa}$ too. By Lebesque's dominated convergence theorem $R_{r}\to 0$ as $ r\to\infty$, which completes the proof.~\qed
\end{proof}

\section{Examples}
The objective of this section is to investigate the Hurst parameter $H$ of the limit process $X_{\kappa}(t)$ in Theorem~\ref{theo4}. The section provides simple examples where the range $\left(\gamma(\Delta),1\right)$ for $H$ is explicitly specified depending on the observation window $\Delta\subset\mathbb{R}^{n}$.

Recall that $H=1-\frac{\kappa\alpha}{2n}-\frac{\beta}{n}$ and  $\mathcal{I}_{1}(\kappa\alpha)$ is defined as
$$\mathcal{I}_{1}(\kappa\alpha)=C\int_{\mathbb{R}^{n}}\dfrac{|K_{\Delta}\left(\lambda\right)|^{2}d\lambda}{\Vert\lambda\Vert^{n-\kappa\alpha-2\beta}}.$$

\begin{example}\label{ex1}
	Let $n=1$ and $\Delta$ has the form $\Delta=[-b,a]\subset\mathbb{R}$, where $a,b\geq0$ and $\vert a+b\vert\neq0$. Using {\rm(\ref{eq3})} one obtains
	$$K_{[-b,a]}\left(\lambda\right)=\int_{-b}^{a}e^{i\lambda x}dx=\dfrac{e^{ia\lambda}-e^{-ib\lambda}}{i\lambda}.$$
	Note, that as $\lambda\to 0$ it holds 
	$|K_{[-b,a]}\left(\lambda\right)|\to b+a < \infty.$
	
	%In this case $\mathcal{I}_{1}(\alpha)<\infty$ if $1-\alpha-2\beta<1$, i.e. $\alpha+2\beta>0$. So, $1-\frac{\alpha}{2}-\beta<1$.
	
	Now, as $\lambda\to\infty$
	$$|K_{[-b,a]}\left(\lambda\right)|=\left|\dfrac{e^{ia\lambda}-e^{-ib\lambda}}{i\lambda}\right|\leq \dfrac{\left|e^{ia\lambda}\right|+\left|e^{-ib\lambda}\right|}{|i\lambda|}=\dfrac{2}{|\lambda|}.$$
	
	Let $\tau_{1}=\tau_{2}=\cdots=\tau_{\kappa}=\alpha$. Then, by~(\ref{eq*}) $\mathcal{I}_{\kappa}(\alpha,\dots,\alpha)$ can be estimated~as
	\begin{align}\label{eq16}
	&\mathcal{I}_{\kappa}(\alpha,\dots,\alpha)\leq C\mathcal{I}_{1}(\kappa\alpha)=C\int_{\mathbb{R}}\dfrac{|K_{[-b,a]}\left(\lambda\right)|^{2}d\lambda}{\vert\lambda\vert^{1-\kappa\alpha-2\beta}}\leq C_{1}\int_{\vert\lambda\vert\leq C_{0}}\dfrac{d\lambda}{\vert\lambda\vert^{1-\kappa\alpha-2\beta}}\notag\\
	&+C_{2}\int_{\vert\lambda\vert> C_{0}}\dfrac{d\lambda}{\vert\lambda\vert^{3-\kappa\alpha-2\beta}}\leq C\int_{0}^{C_{0}}\dfrac{d\rho}{\rho^{1-\kappa\alpha-2\beta}}+C\int_{C_{0}}^{\infty}\dfrac{d\rho}{\rho^{3-\kappa\alpha-2\beta}}.
	\end{align}
	
	Note, that the two conditions $1-\kappa\alpha-2\beta<1$ and $3-\kappa\alpha-2\beta>1$ are required to guarantee that, integrals in~{\rm(\ref{eq16})} are finite. The first condition imples $1-\frac{\kappa\alpha}{2}-\beta<1$ and the second one $1-\frac{\kappa\alpha}{2}-\beta>0$. So $\mathcal{I}_{\kappa}(\alpha,\dots,\alpha)<\infty$ if $H\in(0,1)$. 	 
\end{example}

\begin{example}
	Let $\Delta$ be an $n$-dimensional ball of radius 1, i.e. $\Delta=v(1)\subset\mathbb{R}^{n}$. In this case
	$$K_{v(1)}\left(\lambda\right)=\int_{v(1)}e^{i\langle \lambda,x \rangle}dx.$$
	Note, that as $\Vert\lambda\Vert\to 0$ it holds
	$|K_{v(1)}\left(\lambda\right)|\leq C<\infty.$
	
	Now, as $\Vert\lambda\Vert\to\infty$ we obtain
	$$|K_{v(1)}\left(\lambda\right)|=C\left|\dfrac{J_{n/2}(\Vert\lambda\Vert)}{\Vert\lambda\Vert^{n/2}}\right|<\dfrac{C}{\Vert\lambda\Vert^{\frac{n+1}{2}}}.$$
Let $\tau_{1}=\tau_{2}=\cdots=\tau_{\kappa}=\alpha$. Then, by~(\ref{eq*}) $\mathcal{I}_{\kappa}(\alpha,\dots,\alpha)$ can be estimated~as
	\begin{align}\label{eq17}
	\mathcal{I}_{\kappa}(\alpha,\dots,\alpha)&\leq C\int_{\mathbb{R}^{n}}\dfrac{|K_{v(1)}\left(\lambda\right)|^{2}d\lambda}{\Vert\lambda\Vert^{n-\kappa\alpha-2\beta}}\notag\\&\leq C_{1}\int_{\Vert\lambda\Vert\leq C_{0}}\dfrac{d\lambda}{\Vert\lambda\Vert^{n-\kappa\alpha-2\beta}}+C_{2}\int_{\Vert\lambda\Vert> C_{0}}\dfrac{d\lambda}{\Vert\lambda\Vert^{2n-\kappa\alpha-2\beta+1}}\notag\\
	&\leq C\left[\int_{0}^{C_{0}}\dfrac{d\rho}{\rho^{1-\kappa\alpha-2\beta}}+\int_{C_{0}}^{\infty}\dfrac{d\rho}{\rho^{n-\kappa\alpha-2\beta+2}}\right].
	\end{align}
	
	The two integrals in {\rm(\ref{eq17})} are finite provided that $1-\kappa\alpha-2\beta<1$ and $n-\kappa\alpha-2\beta+2>1$. It follows that $\frac{1}{2}-\frac{1}{2n}<1-\frac{\kappa\alpha}{2n}-\frac{\beta}{n}<1$, i.e. $H\in(\frac{1}{2}-\frac{1}{2n},1)$, where $n\in\mathbb{N}$. Note, that for $n=1$ the Hurst index $H\in(0,1)$ and one obtains the same result as in Example~{\rm\ref{ex1}}. For $n>1$ we get $\gamma(v(1))=\frac{1}{2}-\frac{1}{2n}<\frac{1}{2}$.
\end{example}

\begin{example}
	Let $n=2$, $\Delta=\Box(1)=[-1,1]^{2}\subset\mathbb{R}^{2}$. In this case
	\[
	K_{\Box(1)}\left(\lambda\right)=K_{\Box(1)}\left(\lambda_{1},\lambda_{2}\right)=\int_{-1}^{1}\int_{-1}^{1}e^{i\left(\lambda_{1}x_{1}+\lambda_{2}x_{2}\right)}dx_{1}dx_{2}=\dfrac{\sin\lambda_{1}}{\lambda_{1}}\dfrac{\sin\lambda_{2}}{\lambda_{2}}.
	\]
	
	Note, that $|K_{\Box(1)}\left(\lambda_{1},\lambda_{2}\right)|\leq C$.
	%In this case $\mathcal{I}_{1}(\alpha)<\infty$ if $1-\alpha-2\beta<1$, that is $1-\frac{\alpha}{4}-\frac{\beta}{2}<1$.
	
	When $\min(\lambda_{1},\lambda_{2})>C_{0}>0$ we get
	$$|K_{\Box(1)}\left(\lambda_{1},\lambda_{2}\right)|\leq
	\dfrac{C}{\vert\lambda_{1}\vert\vert\lambda_{2}\vert},$$
	
	and if $\lambda_{j}>C_{0}>0,\ \lambda_{i}\leq C_{0},\  i,j\in\{1,2\},\ i\neq j$, then
	\begin{align}\label{eq17new}
	\sup_{\lambda_{i},i\neq j}|K_{\Box(1)}\left(\lambda_{1},\lambda_{2}\right)|\leq
	\dfrac{C}{\vert\lambda_{j}\vert}.
	\end{align}
	
	Let us split $\mathbb{R}^{2}$ into the regions
	\begin{align*}
	A_{1}^{\prime}:=\{(\lambda_{1},\lambda_{2})\in\mathbb{R}^{2}:\vert\lambda_{1}\vert\leq C_{0},\vert\lambda_{2}\vert\leq C_{0}\},\\
	A_{2}^{\prime}:=\{(\lambda_{1},\lambda_{2})\in\mathbb{R}^{2}:\vert\lambda_{1}\vert\leq C_{0},\vert\lambda_{2}\vert> C_{0}\},\\
	A_{3}^{\prime}:=\{(\lambda_{1},\lambda_{2})\in\mathbb{R}^{2}:\vert\lambda_{1}\vert> C_{0},\vert\lambda_{2}\vert\leq C_{0}\},\\
	A_{4}^{\prime}:=\{(\lambda_{1},\lambda_{2})\in\mathbb{R}^{2}:\vert\lambda_{1}\vert> C_{0},\vert\lambda_{2}\vert> C_{0}\},
	\end{align*}
	where $C_{0}>0$.
	
	Then $\mathcal{I}_{1}(\kappa\alpha)$ can be written as
	\begin{align}\label{eq18}
	\mathcal{I}_{1}(\kappa\alpha)&=\sum_{j=1}^{4}\mathcal{I}_{1}^{(j)}(\kappa\alpha),
	\end{align}
	where $\mathcal{I}_{1}^{(j)}(\kappa\alpha):=\int_{A_{j}^{\prime}}\frac{|K_{\Box(1)}\left(\lambda\right)|^{2}d\lambda}{\Vert\lambda\Vert^{2-\kappa\alpha-2\beta}},\ j=1,\dots, 4$.
	
	We will consider each term in {\rm(\ref{eq18})} separately. The term $\mathcal{I}_{1}^{(1)}(\cdot)$ can be estimated as \begin{align*}
	\mathcal{I}_{1}^{(1)}(\kappa\alpha)&=\int_{A_{1}^{\prime}}\dfrac{|K_{\Box(1)}\left(\lambda\right)|^{2}d\lambda}{\Vert\lambda\Vert^{2-\kappa\alpha-2\beta}}\leq C\int_{A_{1}^{\prime}}\dfrac{d\lambda_{1}d\lambda_{2}}{\left(\vert\lambda_{1}\vert\vert\lambda_{2}\vert\right)^{1-\frac{\kappa\alpha}{2}-\beta}}\notag\\
	&\leq C\left(\int_{\vert\lambda_{1}\vert\leq C_{0}}\dfrac{d\lambda_{1}}{\vert\lambda_{1}\vert^{{1-\frac{\kappa\alpha}{2}-\beta}}}\right)^{2}.
	\end{align*} 
	The last integral is finite provided that $1-\frac{\kappa\alpha}{2}-\beta<1$, i.e. $H<1$.
	
	Using {\rm(\ref{eq17new})} the term $\mathcal{I}_{1}^{(2)}(\cdot)$ can be estimated as
	\begin{align*}
	\mathcal{I}_{1}^{(2)}(\kappa\alpha)&=\int_{A_{2}^{\prime}}\dfrac{|K_{\Box(1)}\left(\lambda\right)|^{2}d\lambda}{\Vert\lambda\Vert^{2-\kappa\alpha-2\beta}}\leq C\int_{A_{2}^{\prime}}\dfrac{d\lambda_{1}d\lambda_{2}}{\vert\lambda_{2}\vert^{2}\left(\vert\lambda_{1}\vert\vert\lambda_{2}\vert\right)^{1-\frac{\kappa\alpha}{2}-\beta}}\notag\\
	&\leq C\int_{\vert\lambda_{1}\vert\leq C_{0}}\dfrac{d\lambda_{1}}{\vert\lambda_{1}\vert^{{1-\frac{\kappa\alpha}{2}-\beta}}}\int_{\vert\lambda_{2}\vert>C_{0}}\dfrac{d\lambda_{2}}{\vert\lambda_{2}\vert^{{3-\frac{\kappa\alpha}{2}-\beta}}}.
	\end{align*} 
	The last integrals are finite provided that $1-\frac{\kappa\alpha}{2}-\beta<1$ and $3-\frac{\kappa\alpha}{2}-\beta>1$. It follows that $H\in(0,1)$. Similarly, one obtains $\mathcal{I}_{1}^{(3)}(\kappa\alpha)<\infty$ when $H\in(0,1)$.
	
	Now, for the term $\mathcal{I}_{1}^{(4)}(\cdot)$ we obtain
	\begin{align*}
	\mathcal{I}_{1}^{(4)}(\kappa\alpha)&=\int_{A_{4}^{\prime}}\dfrac{|K_{\Box(1)}\left(\lambda\right)|^{2}d\lambda}{\Vert\lambda\Vert^{2-\kappa\alpha-2\beta}}\leq C\int_{A_{4}^{\prime}}\dfrac{d\lambda_{1}d\lambda_{2}}{\vert\lambda_{1}\vert^{2}\vert\lambda_{2}\vert^{2}\left(\vert\lambda_{1}\vert\vert\lambda_{2}\vert\right)^{1-\frac{\kappa\alpha}{2}-\beta}}\notag\\
	&\leq C\left(\int_{\vert\lambda_{2}\vert>C_{0}}\dfrac{d\lambda_{2}}{\vert\lambda_{2}\vert^{{3-\frac{\kappa\alpha}{2}-\beta}}}\right)^{2}.
	\end{align*}
	The last integral is finite provided that $3-\frac{\kappa\alpha}{2}-\beta>1$. It follows that $H>0$.
	
	By combining the above results for {\rm(\ref{eq18})}, one obtains $\mathcal{I}_{1}(\kappa\alpha)<\infty$. Therefore, using $\tau_{1}=\tau_{2}=\cdots=\tau_{\kappa}=\alpha$ and the inequality~(\ref{eq*}) we obtain that the result of Theorem~\ref{theo4} is true when $H\in(0,1)$.
\end{example}
%%%%%%%%%%%%%%%%%%%%%%%
\noindent\textbf{Acknowledgements} This research was partially supported under the Australian Research Council's Discovery Projects (project DP160101366).

\end{document}